\documentclass[a4paper,11pt]{amsart}

\usepackage{amsfonts,latexsym,rawfonts,amsmath,amssymb,amsthm, mathrsfs, lscape}
\usepackage[all]{xy}
\usepackage[english]{babel}
\usepackage[utf8]{inputenc}
\usepackage{graphicx}
\usepackage{booktabs}
\usepackage{array, tabularx}
\usepackage{setspace}
\usepackage{textcomp}
\usepackage{tikz}

\usepackage[hypertexnames=false,
backref=page,
    pdftex,
    pdfpagemode=UseNone,
    breaklinks=true,
    extension=pdf,
    colorlinks=true,
    linkcolor=blue,
    citecolor=blue,
    urlcolor=blue,
]{hyperref}

\usepackage[top=1.5in, bottom=.9in, left=0.9in, right=0.9in]{geometry}

\newcommand{\referenza}{}

\newtheorem{thm}{Theorem}[section]
\newtheorem*{thm*}{Theorem \referenza}

\newtheorem*{cor*}{Corollary \referenza}

\newtheorem*{lem*}{Lemma \referenza}

\newtheorem*{prop*}{Proposition \referenza}

\newtheorem*{conj*}{Conjecture \referenza}
\newtheorem{rmk}[thm]{Remark}

\newtheorem{defi}[thm]{Definition}

\numberwithin{equation}{section}

\def \N {\mathbb N}

\def \R {\mathbb R}
\def \C {\mathbb C}
\def \Z {\mathbb Z}

\DeclareMathOperator{\imm}{im}

\allowdisplaybreaks[1]

\title[Cohomological properties of non-K\"ahler manifolds]{Quantitative and qualitative cohomological properties for non-K\"ahler manifolds}

\author{Daniele Angella}
\address[Daniele Angella]{
Dipartimento di Matematica e Informatica ``Ulisse Dini''\\
Universit\`a degli Studi di Firenze\\
viale Morgagni 67/a\\
50134 Firenze, Italy
}
\email{daniele.angella@gmail.com}
\email{daniele.angella@unifi.it}

\author{Nicoletta Tardini}
\address[Nicoletta Tardini]{Dipartimento di Matematica\\
Universit\`a di Pisa\\
largo Bruno Pontecorvo 5\\
56127 Pisa, Italy
}
\email{tardini@mail.dm.unipi.it}

\thanks{During the preparation of this work, the first author was also granted by a Junior Visiting Position at Centro di Ricerca ``Ennio de Giorgi'' in Pisa.
The first author is supported by the Project PRIN ``Varietà reali e complesse: geometria, topologia e analisi armonica'', by the Project FIRB ``Geometria Differenziale e Teoria Geometrica delle Funzioni'', by SNS GR14 grant ``Geometry of non-Kähler manifolds'', by SIR2014 project RBSI14DYEB ``Analytic aspects in complex and hypercomplex geometry'', and by GNSAGA of INdAM.
The second author is supported by GNSAGA of INdAM}

\date{\today}

\begin{document}

\dedicatory{Dedicated to the memory of Professor Pierre Dolbeault.}

\begin{abstract}
 We introduce a ``qualitative property'' for Bott-Chern cohomology of complex non-K\"ahler manifolds, which is motivated in view of the study of the algebraic structure of Bott-Chern cohomology.
 We prove that such a property characterizes the validity of the $\partial\overline\partial$-Lemma.
 This follows from a quantitative study of Bott-Chern cohomology. In this context, we also prove a new bound on the dimension of the Bott-Chern cohomology in terms of the Hodge numbers.
 We also give a generalization of this upper bound, with applications to symplectic cohomologies.
\end{abstract}

\maketitle

\section*{Introduction}

In this note, we investigate quantitative properties of the Bott-Chern and Aeppli cohomologies (namely, relations between their dimensions in terms of the Betti and Hodge numbers) towards the study of their qualitative properties (namely, their algebraic structure). See also \cite{angella-3}, where the results have been announced.

Recall that the Bott-Chern \cite{bott-chern} and Aeppli cohomologies \cite{aeppli} are important tools for the study of the cohomological properties of compact complex (possibly non-K\"ahler) manifolds. They yield, in a sense, a bridge between the de Rham cohomology --- a topological invariant --- and the Dolbeault cohomology --- a holomorphic invariant. In this sense, it is expected that they could provide both a deeper understanding of the holomorphic structure, and natural tools in investigating geometric aspects, see, e.g., \cite{schweitzer, tosatti-weinkove, popovici}.

\medskip

The inequality {\itshape \`a la} Fr\"olicher for the Bott-Chern cohomology in \cite{angella-tomassini-3} provides a lower bound on the dimensions of the Bott-Chern cohomology groups in terms of the Betti numbers. Furthermore, it is proven that this lower bound is reached if and only if the complex manifold satisfies the $\partial\overline\partial$-Lemma, namely, there is a {\em natural} isomorphism between the Bott-Chern and the Aeppli cohomologies. Our first result is an upper bound for the Bott-Chern cohomology in terms of the Hodge numbers. Note that the result is essentially algebraic. At least from the purely algebraic point of view, an upper bound in terms of the Betti numbers is not expected in general.

\renewcommand{\referenza}{\ref{thm:upper-bound} and Remark \ref{rmk:upper-bound}}
\begin{thm*}
 Let $X$ be a compact complex manifold of complex dimension $n$. Then, for any $k\in\Z$,
 \begin{eqnarray*}
 \lefteqn{ \sum_{p+q=k} \dim_\C H^{p,q}_{A}(X) } \\[5pt]
 &\leq& \min\{k+1, (2n-k)+1\} \cdot \left( \sum_{p+q=k} \dim_\C H^{p,q}_{\overline\partial}(X) + \sum_{p+q=k+1} \dim_\C H^{p,q}_{\overline\partial}(X) \right) \\[5pt]
 &\leq& (n+1) \cdot \left( \sum_{p+q=k} \dim_\C H^{p,q}_{\overline\partial}(X) + \sum_{p+q=k+1} \dim_\C H^{p,q}_{\overline\partial}(X) \right) \;,
 \end{eqnarray*}
 and
 \begin{eqnarray*}
 \lefteqn{ \sum_{p+q=k} \dim_\C H^{p,q}_{BC}(X) } \\[5pt]
 &\leq& \min\{k+1, (2n-k)+1\} \cdot \left( \sum_{p+q=k} \dim_\C H^{p,q}_{\overline\partial}(X) + \sum_{p+q=k-1} \dim_\C H^{p,q}_{\overline\partial}(X) \right) \\[5pt]
 &\leq& (n+1) \cdot \left( \sum_{p+q=k} \dim_\C H^{p,q}_{\overline\partial}(X) + \sum_{p+q=k-1} \dim_\C H^{p,q}_{\overline\partial}(X) \right) \;.
 \end{eqnarray*}
\end{thm*}

\medskip

We then get that the difference $\sum_{p+q=k} \left( \dim_\C H^{p,q}_{BC}(X) - \dim_\C H^{p,q}_{A}(X) \right)$ is bounded from both above and below by the Hodge numbers. The second result that we prove is a characterization of the $\partial\overline\partial$-Lemma in terms of this quantity.

\renewcommand{\referenza}{\ref{thm:char-deldelbar-minus}}
\begin{thm*}
 A compact complex manifold $X$ satisfies the $\partial\overline\partial$-Lemma if and only if, for any $k\in\Z$, there holds
 $$ \sum_{p+q=k} \left( \dim_\C H^{p,q}_{BC}(X) - \dim_\C H^{p,q}_{A}(X) \right) \;=\; 0 \;. $$
\end{thm*}

By the Schweitzer duality between the Bott-Chern cohomology and the Aeppli cohomology \cite[\S2.c]{schweitzer}, the above condition can be written just in terms of the Bott-Chern cohomology as follows: for any $k\in\Z$, there holds
$$ \sum_{p+q=k} \dim_\mathbb{C} H^{p,q}_{BC}(X) \;=\; \sum_{p+q=2n-k} \dim_\mathbb{C} H^{p,q}_{BC}(X) \;; $$
compare it with the Poincar\'e duality and the Serre duality.

\medskip

Note that the analogue of the Poincar\'e and Serre dualities in the context of Bott-Chern cohomology is the Schweitzer duality between Bott-Chern and Aeppli cohomologies, see \cite{schweitzer}. 
Namely, the Hermitian duality does not preserve the Bott-Chern cohomology, in general. 
With the aim of studying the algebraic structure of the Bott-Chern cohomology induced from the space of forms, we introduce a property, to which we refer as {\em qualitative Kodaira-Spencer-Schweitzer property}, (compare \cite{kodaira-spencer-3, schweitzer}, where the Bott-Chern Laplacian is introduced and studied). It requires that the natural pairing
$$ H^{\bullet,\bullet}_{BC}(X) \times H^{\bullet,\bullet}_{BC}(X) \to \C \;, \qquad \left( [\alpha], [\beta] \right) \mapsto \int_X \alpha\wedge\beta $$
induced by the wedge product and by the pairing with the fundamental class of $X$ is non-degenerate.
This was initially motivated by attempting to understand an analogue of the Sullivan theory of formality in the context of Bott-Chern cohomology (see also \cite{angella-tomassini-6, tardini-tomassini} for other attempts in this direction).

As a straightforward consequence of the quantitative characterization in Theorem \ref{thm:char-deldelbar-minus}, the qualitative Kodaira-Spencer-Schweitzer property turns out to characterize the $\partial\overline\partial$-Lemma. 

\renewcommand{\referenza}{\ref{thm:main-thm}}
\begin{thm*}
 Let $X$ be a compact complex manifold. Then $X$ satisfies the qualitative Kodaira-Spencer-Schweitzer property if and only if $X$ satisfies the $\partial\overline\partial$-Lemma.
\end{thm*}

We can interpret this qualitative property in view of ``formality'' (in a wider sense) with respect to the duality structure and to the Bott-Chern cohomology functor. In this context, Bott-Chern cohomology reveals itself as an invariant strong enough to force natural isomorphisms; compare also \cite[Theorem B]{angella-tomassini-3}.

\medskip

Finally, we state a similar upper bound for a more general class of double complexes in Theorem \ref{thm:upper-bound-algebraic}. This allows to get results about the symplectic cohomologies introduced and studied by L.-S. Tseng and S.-T. Yau \cite{tseng-yau-1, tseng-yau-2}.

\renewcommand{\referenza}{\ref{thm:upper-bound-symplectic}}
\begin{thm*}
 Let $X$ be a compact differentiable manifold of dimension $2n$ endowed with a symplectic structure $\omega$. Then, for any $k\in\Z/2\Z$,
 $$
 \sum_{h = k \,\mathrm{mod}\, 2} \dim_\R H^{h}_{d+d^\Lambda}(X)
 \;\leq\;
 2(2n+1) \cdot \sum_{h\in\Z} \dim_\R H^{h}_{dR}(X;\R) \;,
 $$
 and
 $$
 \sum_{h = k \,\mathrm{mod}\, 2} \dim_\R H^{h}_{dd^\Lambda}(X)
 \;\leq\;
 2(2n+1) \cdot \sum_{h\in\Z} \dim_\R H^{h}_{dR}(X;\R) \;.
 $$
\end{thm*}

\bigskip

\noindent{\sl Acknowledgments.}
Part of this work has been written in occasion of the Workshop ``Bielefeld Geometry \& Topology Days'' held at Bielefeld University on July 2nd--3rd, 2015: the first author warmly thanks the organizer Giovanni Bazzoni for the kind invitation and hospitality. We acknowledge many important discussions with Adriano Tomassini on the subject. Thanks also to Giovanni Bazzoni and Michela Zedda for interesting conversations and useful suggestions, and to Valentino Tosatti and Jim Stasheff for helpful comments that improved the presentation of the paper.

\section{Preliminaries and notations}
Let $X$ be a compact complex manifold. We recall that the {\em Bott-Chern cohomology} \cite{bott-chern} is defined as
$$ H_{BC}^{\bullet,\bullet}(X) \;=\; \frac{\ker\partial\cap\ker\overline\partial}{\imm\partial\overline\partial} \;, $$
and the {\em Aeppli cohomology} \cite{aeppli} is defined as
$$ H_{A}^{\bullet,\bullet}(X) \;=\; \frac{\ker\partial\overline\partial}{\imm\partial+\imm\overline\partial} \;. $$
The same definitions can be stated, more generally, for a double complex $\left( B^{\bullet,\bullet}, \partial, \overline\partial \right)$ of vector spaces.

In \cite{schweitzer}, see also \cite{kodaira-spencer-3}, Hodge theory for the Bott-Chern and Aeppli cohomologies is developed. In particular, it follows that the Bott-Chern and Aeppli cohomology groups of compact complex manifolds are finite-dimensional vector spaces. Moreover, when a Hermitian metric is fixed on $X$, the $\C$-linear Hodge-$*$-operator induces an (un-natural) isomorphism between the Bott-Chern cohomology and the Aeppli cohomology.

The identity induces natural maps of (bi-)graded vector spaces between the Bott-Chern, Dolbeault, de Rham, and Aeppli cohomologies:
$$ \xymatrix{
  & H^{\bullet,\bullet}_{BC}(X) \ar[d]\ar[ld]\ar[rd] & \\
  H^{\bullet,\bullet}_{\partial}(X) \ar[rd] & H^{\bullet}_{dR}(X;\C) \ar[d] & H^{\bullet,\bullet}_{\overline\partial}(X) \ar[ld] \\
  & {\phantom{\;.}} H^{\bullet,\bullet}_{A}(X) \;. &
} $$
By definition, a compact complex manifold satisfies the {\em $\partial\overline\partial$-Lemma} if the natural map $H^{\bullet,\bullet}_{BC}(X)\to H^{\bullet,\bullet}_{A}(X)$ is injective. This is equivalent to any of the above maps being an isomorphism, see \cite[Lemma 5.15]{deligne-griffiths-morgan-sullivan}.

\medskip

We recall two results concerning a comparison of the dimensions of the above cohomologies. The first one is the classical {\em Fr\"olicher inequality}, \cite[Theorem 2]{frolicher}. It states that there is a spectral sequence of the form
$$ H^{\bullet,\bullet}_{\overline\partial}(X) \Rightarrow H^\bullet_{dR}(X;\C) \;, $$
whence the inequality, for any $k\in\Z$,
$$ \sum_{p+q=k} \dim_\C H^{p,q}_{\overline\partial}(X) - b_k \;\geq\; 0 \;, $$
where $b_k$ denotes the $k$-th Betti number of $X$.
(Sometimes, for clearness of notation, we will shorten, e.g., $h^k_A:=\sum_{p+q=k}\dim_\C H^{p,q}_{A}(X)$.)

In \cite{angella-tomassini-3}, an inequality {\itshape à la} Fr\"olicher for the Bott-Chern cohomology is studied. It provides a lower bound for the dimension of the Bott-Chern and Aeppli cohomologies in terms of the Betti numbers, and it yields also a characterization of the $\partial\overline\partial$-Lemma. 

\begin{thm}[{\cite[Theorem A, Theorem B]{angella-tomassini-3}}]
 Let $X$ be a compact complex manifold. Then, for any $k\in\Z$,
 $$ \Delta^k(X) \;:=\; \sum_{p+q=k} \left( \dim_\C H^{p,q}_{BC}(X) + \dim_\C H^{p,q}_{A}(X) \right) - 2\, b_k \;\geq\; 0 \;. $$
 Moreover, $X$ satisfies the $\partial\overline\partial$-Lemma if and only if, for any $k\in\Z$, there holds $\Delta^k(X)=0$.
\end{thm}

See \cite{angella-tomassini-5} for a generalization to double complexes, with applications to compact symplectic manifolds.

\section{An upper bound on the dimension of Bott-Chern cohomology for compact complex manifolds}
The result in \cite{angella-tomassini-3} provides a lower bound for the dimension of the Bott-Chern cohomology in terms of the Betti numbers; furthermore, it characterizes the $\partial\overline\partial$-Lemma.
In this section, we prove an upper bound for the dimension of the Bott-Chern cohomology in terms of the Hodge numbers. Note that a topological upper bound cannot be expected, in general.

\medskip

The result we prove is the following.
\begin{thm}\label{thm:upper-bound}
 Let $X$ be a compact complex manifold of complex dimension $n$. Then, for any $k\in\Z$,
 $$ \sum_{p+q=k} \dim_\C H^{p,q}_{A}(X) \;\leq\; (n+1) \cdot \left( \sum_{p+q=k} \dim_\C H^{p,q}_{\overline\partial}(X) + \sum_{p+q=k+1} \dim_\C H^{p,q}_{\overline\partial}(X) \right) \;. $$
\end{thm}

\begin{rmk}\label{rmk:upper-bound}
 In fact, the constant $n+1$ in the statement can be replaced by $\min\{k+1, (2n-k)+1\}$: for any $k\in\Z$,
 \begin{eqnarray}\label{eq:upper-bound}
 \lefteqn{ \sum_{p+q=k} \dim_\C H^{p,q}_{A}(X) } \\[5pt]
 &\leq& \min\{k+1, (2n-k)+1\} \cdot \left( \sum_{p+q=k} \dim_\C H^{p,q}_{\overline\partial}(X) + \sum_{p+q=k+1} \dim_\C H^{p,q}_{\overline\partial}(X) \right) \;. \nonumber 
 \end{eqnarray}
 Moreover, by using the Schweitzer duality \cite[\S2.c]{schweitzer} and the Serre duality, we get an analogue inequality for the Bott-Chern cohomology: for any $k\in\Z$,
 \begin{eqnarray*}
 \lefteqn{ \sum_{p+q=k} \dim_\C H^{p,q}_{BC}(X) } \\[5pt]
 &\leq& \min\{k+1, (2n-k)+1\} \cdot \left( \sum_{p+q=k} \dim_\C H^{p,q}_{\overline\partial}(X) + \sum_{p+q=k-1} \dim_\C H^{p,q}_{\overline\partial}(X) \right) \\[5pt]
 &\leq& (n+1) \cdot \left( \sum_{p+q=k} \dim_\C H^{p,q}_{\overline\partial}(X) + \sum_{p+q=k-1} \dim_\C H^{p,q}_{\overline\partial}(X) \right) \;. \nonumber 
 \end{eqnarray*}
 \end{rmk}

\begin{proof}[Proof (of Theorem \ref{thm:upper-bound}).]
 Let $0\neq[\alpha^{p,q}]\in H^{p,q}_A(X)$ with $p+q=k$. In particular, $\partial\overline\partial\alpha^{p,q}=0$, and $\alpha^{p,q}$ is neither $\partial$-exact nor $\overline\partial$-exact (in fact, it does not belong to $\imm\partial+\imm\overline\partial$). We distinguish four possible cases.
 \begin{enumerate}
  \item {\itshape It holds $\partial\alpha^{p,q}=\overline\partial\alpha^{p,q}=0$.}
  Then, by setting $\alpha_1:=\alpha^{p,q}$ and $\alpha_2:=\alpha^{p,q}$, we get two non-trivial (conjugate-)Dolbeault classes $0\neq[\alpha_1]\in H^{p,q}_{\overline\partial}(X)$ and $0\neq[\alpha_2]\in H^{p,q}_{\partial}(X)$.
  \item {\itshape It holds $\partial\alpha^{p,q}=0$ and $\overline\partial\alpha^{p,q}\neq0$.} By setting $\alpha_2:=\alpha^{p,q}$, we get a non-trivial conjugate-Dolbeault class $0\neq[\alpha_2]\in H^{p,q}_{\partial}(X)$. Consider now $\gamma^{p,q+1}:=\overline\partial\alpha^{p,q}$. Notice that $\partial\gamma^{p,q+1}=\partial\overline\partial\alpha^{p,q}=0$. We distinguish two cases.
  \begin{enumerate}
   \item {\itshape The form $\gamma^{p,q+1}$ is not $\partial$-exact.} In this case, set $\alpha_1:=\gamma^{p,q+1}$, getting a non-trivial class in conjugate-Dolbeault cohomology $0\neq[\alpha_1]\in H^{p,q+1}_{\partial}(X)$.
   \item {\itshape The form $\gamma^{p,q+1}$ is $\partial$-exact.} Let $\alpha^{p-1,q+1}$ be such that $\gamma^{p,q+1}=\partial\alpha^{p-1,q+1}$. (See Figure \ref{fig:proof}.) Notice that, without loss of generality, we may assume $0\neq[\alpha^{p-1,q+1}]\in H^{p-1,q+1}_{A}(X)$. Indeed, first notice that $\partial\overline\partial\alpha^{p-1,q+1}=-\overline\partial\gamma^{p,q+1}=-\overline\partial\overline\partial\alpha^{p,q}=0$. Moreover, $\alpha^{p-1,q+1}$ is not $\partial$-exact, otherwise $\gamma^{p,q+1}=\partial\alpha^{p-1,q+1}$ is zero. Finally, if $\alpha^{p-1,q+1}$ is $\overline\partial$-exact, let us say $\alpha^{p-1,q+1}=\overline\partial\beta^{p-1,q}$, then, up to consider $[\alpha^{p,q}]=[\alpha^{p,q}+\partial\beta^{p-1,q}]\in H^{p,q}_{A}(X)$, we are in case (1), since $\partial(\alpha^{p,q}+\partial\beta^{p-1,q})=\overline\partial(\alpha^{p,q}+\partial\beta^{p-1,q})=0$.
   Moreover, if $\alpha^{p-1,q+1}=\partial\lambda^{p-2,q+1}+\overline\partial\mu^{p-1,q}$ then, up to consider $[\alpha^{p-1,q+1}]=[\alpha^{p-1,q+1}-\partial\lambda^{p-2,q+1}]\in H^{p-1,q+1}_{A}(X)$ we still have $\partial\left(\alpha^{p-1,q+1}-\partial\lambda^{p-2,q+1}\right)=\partial\alpha^{p-1,q+1}=\gamma^{p,q+1}$ but now $\alpha^{p-1,q+1}-\partial\lambda^{p-2,q+1}$ is $\overline\partial$-exact so we are in case (1) as noticed above.\\
   Now, we still distinguish two cases. If $\overline\partial\alpha^{p-1,q+1}=0$, then set $\alpha_1:=\alpha^{p-1,q+1}$ and consider $0\neq[\alpha_1]\in H^{p-1,q+1}_{\overline\partial}(X)$. Otherwise, let $\gamma^{p-1,q+2}:=\overline\partial\alpha^{p-1,q+1}$, and go on again with the two cases (2.a) and (2.b).
      \begin{figure}
      \label{fig:proof}
      \begin{center}
      \begin{tikzpicture}
      \newcommand\un{1.5}

      \draw[help lines, step=\un] (0,0) grid (4*\un,4*\un);

      \node at (\un*.5+\un*1,-.5) {$p-1$};
      \node at (\un*.5+\un*2,-.5) {$p$};
      \node at (-.5,\un*.5+\un*1) {$q$};
      \node at (-.5,\un*.5+\un*2) {$q+1$};

      \coordinate (A) at (2*\un+1/2*\un, 1*\un+1/2*\un);
      \coordinate (B) at (2*\un+1/2*\un, 2*\un+1/2*\un);
      \coordinate (C) at (1*\un+1/2*\un, 2*\un+1/2*\un);

      \newcommand{\raggio}{1*\un pt}
      \fill (A) circle (\raggio);
      \fill (B) circle (\raggio);
      \fill (C) circle (\raggio);

      \draw (A) -- (B);
      \draw (C) -- (B);

      \begingroup\makeatletter\def\f@size{6}\check@mathfonts
      \node at (2*\un+1/2*\un, 1*\un+1/2*\un-.3) {$\alpha^{p,q}$};
      \node at (2*\un+1/2*\un, 2*\un+1/2*\un+.3) {$\gamma^{p,q+1}$};
      \node at (1*\un+1/2*\un, 2*\un+1/2*\un+.3) {$\alpha^{p-1,q+1}$};
      \endgroup

      \end{tikzpicture}
      \end{center}
      \caption{Case 2.b in the Proof of Theorem \ref{thm:upper-bound}.}
      \end{figure}
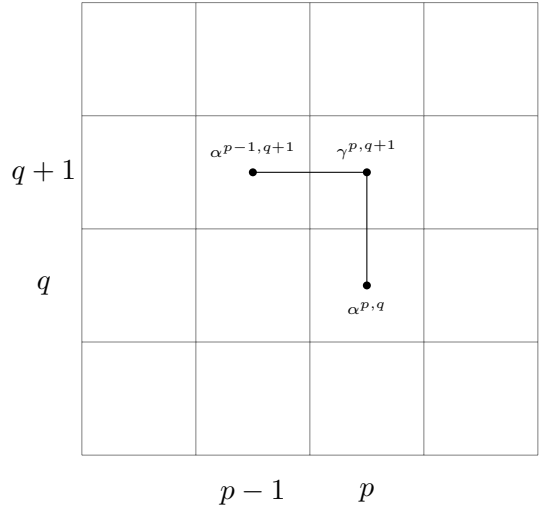
   \end{enumerate}
   Note that this process allows to find a non-trivial class, either in Dolbeault cohomology at total degree $k$ or in conjugate-Dolbeault cohomology at total degree $k+1$, in a finite number of steps. This number is at most the length of the anti-diagonal zigzag at degree $k$, namely, $\min\{k+1,(2n-k)+1\}$.
  \item {\itshape There holds $\partial\alpha^{p,q}\neq0$ and $\overline\partial\alpha^{p,q}=0$.} This case is analogous to the case (2), by symmetry.
  \item {\itshape There holds $\partial\alpha^{p,q}\neq0$ and $\overline\partial\alpha^{p,q}\neq0$.} Set $\gamma^{p,q+1}:=\overline\partial\alpha^{p,q}$, and proceed as in (2), distinguishing between cases (2.a) and (2.b) to get a non-trivial class either in Dolbeault cohomology at total degree $k$ or in conjugate-Dolbeault cohomology at total degree $k+1$. Consider then $\gamma^{p+1,q}:=\partial\alpha^{p,q}$, and proceed as in (3), by symmetry, distinguishing between two cases to get a non-trivial class either in conjugate-Dolbeault cohomology at total degree $k$ or in Dolbeault cohomology at total degree $k+1$.
 \end{enumerate}
 Summarizing, we have that any of the classes given by $[\alpha^{p+j,q-j}]\in H^{p+j,q-j}_{A}(X)$, which are at most $\min\{k+1, (2n-k)+1\}$ many, yields two non-trivial classes in either Dolbeault or conjugate-Dolbeault, either at degree $k$ or at degree $k+1$.
 
 In other words, consider the space of zigzags, namely,
 $$ \mathcal{Z} \;:=\; \left\{
 \left\{ \left[\alpha^{p+j+\delta,q-j}\right]
 \right\}_{j\in\Z,\delta\in\{0,1\}}
 \in
 \bigoplus_{j\in\Z} \left( H^{p+j,q-j}_{A}(X) \oplus H^{p+j+1,q-j}_{BC}(X) \right)
 \right\} \;, $$
 where $\alpha^{p+j,q-j}$ and $\alpha^{p+j+1,q-j}:=\gamma^{p+j+1,q-j}$ are constructed as above and extended with $0$.
 Note that we can choose a basis $\mathcal{B}$ for this linear space with the property that its elements have linear-independent ends, (that is, the first and last non-zero classes, $[\alpha_1]$ and $[\alpha_2]$ respectively,) in the sum of Dolbeault and conjugate-Dolbeault cohomologies.
 
 At the end, we get
 \begin{eqnarray*}
  h^k_{A} &\leq& \min\{k+1, (2n-k)+1\} \cdot \left( h^k_{\overline\partial} + h^{k+1}_{\overline\partial} \right) \\[5pt]
  &\leq& (n+1) \cdot \left( h^k_{\overline\partial} + h^{k+1}_{\overline\partial} \right) \;,
 \end{eqnarray*}
 proving the statement.
\end{proof}

\section{A quantitative characterization of \texorpdfstring{$\partial\overline\partial$}{partiaoverlinepartial}-Lemma}

As a consequence of Theorem \ref{thm:upper-bound} and by \cite[Theorem A]{angella-tomassini-3}, we get the following bounds: for any $k\in\Z$,
\begin{eqnarray*}
2 \, h^k_{\overline\partial} \;\;\leq &h^k_{A}+h^{k}_{BC}& \leq\;\; (n+1) \cdot \left( h^{k-1}_{\overline\partial} + 2\, h^k_{\overline\partial} + h^{k+1}_{\overline\partial} \right) \;, \\[5pt]
-2(n+1) \cdot \left( h^{k-1}_{\overline\partial} + h^{k}_{\overline\partial} \right) \;\;\leq &h^k_{A} - h^{k}_{BC}& \leq\;\; 2(n+1) \cdot \left( h^k_{\overline\partial} + h^{k+1}_{\overline\partial} \right) \;.
\end{eqnarray*}

By \cite[Theorem B]{angella-tomassini-3}, the quantity $h^{k}_{BC}+h^{k}_{A}$ yields a characterization of the $\partial\overline\partial$-Lemma. In this section, we provide a characterization of the $\partial\overline\partial$-Lemma in terms of the quantity $h^{k}_{BC}-h^{k}_{A}$.

\medskip

We prove the following result.

\begin{thm}\label{thm:char-deldelbar-minus}
 A compact complex manifold $X$ satisfies the $\partial\overline\partial$-Lemma if and only if, for any $k\in\Z$, there holds
 $$ \sum_{p+q=k} \left( \dim_\C H^{p,q}_{BC}(X) - \dim_\C H^{p,q}_{A}(X) \right) \;=\; 0 \;. $$
\end{thm}

\begin{proof}
 The ``only if'' part being trivial, we need just to prove that the numerical equality $h^k_{BC}=h^k_{A}$ for any $k\in\Z$ yields the $\partial\overline\partial$-Lemma.
 
 Consider the J. Varouchas \cite{varouchas} exact sequences. We recall the following definitions:
 \begin{multline*}
 A^{\bullet,\bullet} \;:=\; \frac{\imm\overline\partial\cap\imm\partial}{\imm\partial\overline\partial} \;, \qquad
 B^{\bullet,\bullet} \;:=\; \frac{\ker\overline\partial\cap\imm\partial}{\imm\partial\overline\partial} \;, \qquad
 C^{\bullet,\bullet} \;:=\; \frac{\ker\partial\overline\partial}{\ker\overline\partial+\imm\partial} \;,
 \\[5pt]
 D^{\bullet,\bullet} \;:=\; \frac{\imm\overline\partial\cap\ker\partial}{\imm\partial\overline\partial} \;, \qquad
 E^{\bullet,\bullet} \;:=\; \frac{\ker\partial\overline\partial}{\ker\partial+\imm\overline\partial} \;, \qquad
 F^{\bullet,\bullet} \;:=\; \frac{\ker\partial\overline\partial}{\ker\overline\partial+\ker\partial} \;.
 \end{multline*}
 Then, by \cite[Section 3.1]{varouchas}, one has that the following sequences of finite-dimensional vector spaces are exact:
 \begin{multline*}
 0 \to A^{\bullet,\bullet} \to B^{\bullet,\bullet} \to H^{\bullet,\bullet}_{\overline\partial}(X) \to H^{\bullet,\bullet}_{A}(X) \to C^{\bullet,\bullet} \to 0 \;, \\[5pt]
 0 \to D^{\bullet,\bullet} \to H^{\bullet,\bullet}_{BC}(X) \to H^{\bullet,\bullet}_{\overline\partial}(X) \to E^{\bullet,\bullet} \to F^{\bullet,\bullet} \to 0 \;.
 \end{multline*}

 We denote by small romanic letters, $a^{p,q}:=\dim_\C A^{p,q}$, \dots, $f^{p,q}:=\dim_\C F^{p,q}$, the dimensions of the corresponding groups. Moreover, we denote $a^k:=\sum_{p+q=k}a^{p,q}$, \dots, $f^k:=\sum_{p+q=k}f^{p,q}$. By conjugation and by the isomorphisms $\partial\colon E^{\bullet,\bullet} \stackrel{\simeq}{\to} B^{\bullet+1,\bullet}$ and $\overline\partial\colon C^{\bullet,\bullet} \stackrel{\simeq}{\to} D^{\bullet,\bullet+1}$, we have, in particular,
 $$ d^k \;=\; b^k \;, \qquad e^k \;=\; c^k \;, \qquad c^k \;=\; b^{k+1} \;. $$

 From the exact sequences above, for any $k\in\Z$, we have
 \begin{eqnarray*}
  h^{k}_{BC} - h^{k}_{A}
  &=& d^k-e^k+f^k-c^k+b^k-a^k \\[5pt]
  &=& 2 b^k - 2 c^k + f^k - a^k \\[5pt]
  &=& 2 b^k - 2 b^{k+1} + f^k - a^k \;.
 \end{eqnarray*}
 
 By the hypothesis that the left-hand side is zero, we get
 $$ 2 b^{k+1} \;=\; 2 b^k + f^k - a^k \;, $$
 whence (note that $b^0=0$)
 $$ b^{k+1} \;=\; \frac{1}{2} \sum_{\ell=0}^{k} \left( f^\ell - a^\ell \right) \;. $$
 
 We argue now by induction on $k\in\N$ to prove that $a^k=0$ and $b^k=0$. This is enough to prove that $X$ satisfies the $\partial\overline\partial$-Lemma: indeed, also $c^k=b^{k+1}=0$ for any $k\in\Z$; whence we get that $H^{\bullet,\bullet}_{\overline\partial}(X)\to H^{\bullet,\bullet}_{A}(X)$ is an isomorphism; and therefore, by \cite[Lemma 5.15]{deligne-griffiths-morgan-sullivan}, we get the thesis. For $k=0$, we have $0\leq a^1\leq b^1\leq \frac{1}{2} \left( f^0-a^0 \right)=0$. We suppose by induction that $a^\ell=0=b^\ell$ for any $\ell\in\{0,\dots,k\}$. By the exactness of $E^{\bullet,\bullet}\to F^{\bullet,\bullet}\to 0$, we have $f^\ell\leq e^\ell=c^\ell=b^{\ell+1}$ for any $\ell\in\Z$. In particular, also $f^\ell=0$ for any $\ell\in\{0, \ldots, k-1\}$. We then have $0\leq a^{k+1}\leq b^{k+1}=\frac{1}{2}\sum_{\ell=0}^{k}\left( f^\ell-a^\ell \right)=\frac{1}{2} f^k\leq \frac{1}{2} b^{k+1}$ whence also $a^{k+1}=b^{k+1}=0$, proving the claim.
\end{proof}

\section{Examples}

In Table \ref{table:examples-degrees}, we summarize, for $k\in\Z$, the quantities
$$ S^k \;:=\; \min \{ k+1, (2n-k)+1 \} \cdot \left( h^k_{\overline\partial} + h^{k+1}_{\overline\partial} \right) - h^k_A \;\in\; \N $$
and
$$ N^k \;:=\; h^k_{A} - h^k_{BC} \;\in\; \Z $$
on some explicit examples.
We also report the quantity
$$ \Delta^k \;:=\; h^k_A+h^k_{BC}-2b_k \;\in\; \N \;. $$
More precisely, we consider the Iwasawa manifold and its small deformations, for which the Bott-Chern cohomology is computed in \cite{schweitzer, angella-1}, some compact complex surfaces using the results in \cite{angella-dloussky-tomassini}, and the Nakamura manifolds of completely-solvable, respectively holomorphically-parallelizable type, and some small deformations of it, as studied in \cite{angella-kasuya-1}. In fact, for suitable compact quotients of nilpotent or solvable Lie groups, the Bott-Chern cohomology can be computed by restricting to a finite-dimensional sub-complex of the complex of forms. Similar results can be obtained for small deformations.

\begin{center}
\begin{table}[ht]
 \centering
\resizebox{\textwidth}{!}{
 \begin{tabular}{>{\bfseries\bgroup}l<{\bfseries\egroup} || >{$}c<{$} >{$}c<{$} >{$}c<{$} || >{$}c<{$} >{$}c<{$} >{$}c<{$} || >{$}c<{$} >{$}c<{$} >{$}c<{$} || >{$}c<{$} >{$}c<{$} >{$}c<{$} || >{$}c<{$} >{$}c<{$} >{$}c<{$} ||}
\toprule
manifold &
S^1 & N^1 & \Delta^1 & 
S^2 & N^2 & \Delta^2 & 
S^3 & N^3 & \Delta^3 & 
S^4 & N^4 & \Delta^4 & 
S^5 & N^5 & \Delta^5
\\
\toprule
Inoue $S_M$ &
0 & 2 & 0 & 
2 & 0 & 2 & 
4 & -2 & 0 & 
- & - & - & 
- & - & -
\\
primary Kodaira &
10 & 2 & 0 & 
16 & 0 & 2 & 
6 & -2 & 0 & 
- & - & - & 
- & - & -
\\
secondary Kodaira &
0 & 2 & 0 & 
2 & 0 & 2 & 
4 & -2 & 0 & 
- & - & - & 
- & - & -
\\
Inoue $S^\pm$ &
0 & 2 & 0 & 
2 & 0 & 2 & 
4 & -2 & 0 & 
- & - & - & 
- & - & -
\\
Calabi-Eckmann $\mathbb{S}^1\times\mathbb{S}^3$ &
0 & 2 & 0 & 
2 & 0 & 2 & 
4 & -2 & 0 & 
- & - & - & 
- & - & -
\\
\midrule
Iwasawa and deformations (i) &
26 & 2 & 2 & 
63 & 2 & 6 & 
86 & 0 & 8 & 
38 & -2 & 6 & 
8 & -2 & 2
\\
Iwasawa deformations (ii.a) &
20 & 2 & 2 & 
52 & 3 & 3 & 
70 & 0 & 8 & 
31 & -3 & 3 & 
6 & -2 & 2
\\
Iwasawa deformations (ii.b) &
20 & 2 & 2 & 
53 & 2 & 2 & 
70 & 0 & 8 & 
31 & -2 & 2 & 
6 & -2 & 2
\\
Iwasawa deformations (iii.a) &
18 & 2 & 2 & 
43 & 5 & 1 & 
58 & 0 & 8 & 
30 & -5 & 1 & 
6 & -2 & 2
\\
Iwasawa deformations (iii.b) &
18 & 2 & 2 & 
44 & 4 & 0 & 
58 & 0 & 8 & 
30 & -4 & 0 & 
6 & -2 & 2
\\
\midrule
Nakamura completely-solvable (i) &
32 & 8 & 8 & 
88 & 4 & 20 & 
120 & 0 & 24 & 
50 & -4 & 20 & 
12 & -8 & 8
\\
Nakamura completely-solvable (ii) &
16 & 0 & 0 & 
48 & 4 & 4 & 
64 & 0 & 8 & 
22 & -4 & 4 & 
4 & 0 & 0
\\
Nakamura completely-solvable (iii) &
12 & 0 & 0 & 
34 & 0 & 0 & 
44 & 0 & 0 & 
16 & 0 & 0 & 
4 & 0 & 0
\\
Nakamura holomorphically-parallelizable (a) &
32 & 8 & 8 & 
88 & 4 & 20 & 
120 & 0 & 24 & 
50 & -4 & 20 & 
12 & -8 & 8
\\
Nakamura holomorphically-parallelizable (b) &
16 & 4 & 4 & 
38 & 0 & 8 & 
52 & 0 & 8 & 
26 & 0 & 8 & 
8 & -4 & 4
\\
\bottomrule
\end{tabular}
}
\caption{The degrees $S^k$, $N^k$, and $\Delta^k$ on some non-K\"ahler examples.}
\label{table:examples-degrees}
\end{table}
\end{center}

\begin{rmk}
In particular, notice that the value of $S^k$ may be zero: so the inequality in \eqref{eq:upper-bound} is sharp.
\end{rmk}

\section{A qualitative characterization of \texorpdfstring{$\partial\overline\partial$}{partialoverlinepartial}-Lemma}

In this section, we are ultimately aimed at understanding the algebraic structure induced on the Bott-Chern cohomology by the structure of the space of differential forms. Initially motivated by understanding whether it is possible to mimic the Sullivan theory of formality \cite{sullivan} in the context of Bott-Chern cohomology, (see some first attempts in \cite{angella-tomassini-6, tardini-tomassini},) we introduce and study a ``qualitative property'', which is proven to characterize the $\partial\overline\partial$-Lemma. This is a consequence of the results in the previous sections.

\medskip

Note that the triple Aeppli-Bott-Chern Massey products introduced in \cite{angella-tomassini-6} take values in the Aeppli cohomology groups starting from Bott-Chern classes. Recall that, in the Sullivan theory of formality for the de Rham cohomology \cite{sullivan, deligne-griffiths-morgan-sullivan}, (compare also the Neisendorfer and Taylor theory of Dolbeault-formality for the Dolbeault cohomology \cite{neisendorfer-taylor},) the Massey products are related with the $A_\infty$-algebra structure induced on any deformation retract by the Homotopy Transfer Principle, \cite{lu-palmieri-wu-zhang}. In view of understanding a possible analogue notion of $A_\infty$-algebra in the context of the Bott-Chern cohomology, one would first attempt to construct higher-order Massey products. When using the definition in \cite{angella-tomassini-6}, this requires to look at Aeppli classes as Bott-Chern classes.
This suggests us to investigate compact complex manifolds satisfying the following ``qualitative'' property.

\begin{defi}
 A compact complex manifold $X$ is said to satisfy the {\em qualitative Kodaira-Spencer-Schweitzer property} if the natural pairing
 $$ H^{\bullet,\bullet}_{BC}(X) \times H^{\bullet,\bullet}_{BC}(X) \to \C \;, \qquad \left( [\alpha], [\beta] \right) \mapsto \int_X \alpha\wedge\beta $$
 induced by the wedge product and by the pairing with the fundamental class of $X$ is non-degenerate.
\end{defi}

We have in fact that this property characterizes the $\partial\overline\partial$-Lemma.

\begin{thm}\label{thm:main-thm}
 Let $X$ be a compact complex manifold. Then $X$ satisfies the qualitative Kodaira-Spencer-Schweitzer property if and only if $X$ satisfies the $\partial\overline\partial$-Lemma.
\end{thm}

\begin{proof}
 For the ``only if'' part: by the hypothesis, we get that the Bott-Chern cohomology is isomorphic to its dual, which is in turn (non-naturally) isomorphic to the Aeppli cohomology by \cite{schweitzer}. Therefore the statement follows from Theorem \ref{thm:char-deldelbar-minus}.
 
 For the ``if'' part: fix a Hermitian metric $g$ on $X$. Consider $*$ the $\C$-linear Hodge-star-operator associated to $g$, and let $\tilde\Delta_{BC}$ and $\tilde\Delta_{A}$ denote the Bott-Chern, respectively Aeppli Laplacians \cite{kodaira-spencer-3, schweitzer} with respect to $g$.
 Let $[\alpha_{h_{BC}}]_{BC}$ be a Bott-Chern cohomology class, where $\alpha_{h_{BC}}$ is the $\tilde\Delta_{BC}$-harmonic representative. Consider the Aeppli cohomology class $[\overline{*}\alpha_{h_{BC}}]_{A}$. By the $\partial\overline\partial$-Lemma, the natural map $H^{\bullet,\bullet}_{BC}(X) \to H^{\bullet,\bullet}_{A}(X)$ is surjective \cite[Lemma 5.15]{deligne-griffiths-morgan-sullivan}, whence there exist $\gamma$ and $\eta$ forms such that $\overline{*}\alpha_{h_{BC}}+\partial\gamma+\overline\partial\eta\in\ker\partial\cap\ker\overline\partial$. Therefore we have the Bott-Chern cohomology class $[\overline{*}\alpha_{h_{BC}}+\partial\gamma+\overline\partial\eta]_{BC}$. We have
 \begin{eqnarray*}
  \left( [\alpha_{h_{BC}}]_{BC},\, [\overline{*}\alpha_{h_{BC}}+\partial\gamma+\overline\partial\eta]_{BC} \right)
  &=& \int_X \alpha_{h_{BC}} \wedge \left( \overline{*}\alpha_{h_{BC}}+\partial\gamma+\overline\partial\eta \right) \\[5pt]
  &=& \int_X \alpha_{h_{BC}} \wedge \overline{*}\alpha_{h_{BC}} \;,
 \end{eqnarray*}
 proving the statement.
\end{proof}

\begin{rmk}
 For a non-K\"ahler example of compact complex manifolds satisfying the Kodaira-Spencer-Schweitzer qualitative property, consider the completely-solvable Nakamura manifold in case {\itshape (iii)} in \cite[Example 2.17]{angella-kasuya-1}, see \cite[Table 3]{angella-kasuya-1}.
\end{rmk}

\section{An algebraic generalization with application to compact symplectic manifolds}

\subsection{Algebraic upper bound}
We note that the upper bound in Theorem \ref{thm:upper-bound} can be slightly generalized to a more general context, similar as done in \cite{angella-tomassini-5} for the inequality {\itshape \`a la} Fr\"olicher.

\begin{thm}\label{thm:upper-bound-algebraic}
 Let $\left(B^{\bullet,\bullet}, \partial, \overline\partial\right)$ be a (possibly unbounded) double complex of $\mathbb{K}$-vector spaces.
 Suppose that there exists $N\in\N$, there exists $\ell\in\{0,\ldots,N\}$, such that for any $q\in\Z$, there holds $B^{p,q}=\{0\}$ for $p\not\in\{\ell\cdot q, \ldots, \ell\cdot q+N\}$.
 Then, for any $k\in\Z$,
 \begin{eqnarray}\label{eq:upper-bound-aeppli-algebraic}
 \lefteqn{ \sum_{p+q=k} \dim_\mathbb{K} H^{p,q}_{A}(B^{\bullet,\bullet}) } \\[5pt]
 &\leq& (N+1) \cdot \left( \sum_{p+q=k} \left( \dim_\mathbb{K} H^{p,q}_{\partial}(B^{\bullet,\bullet}) + \dim_\mathbb{K} H^{p,q}_{\overline\partial}(B^{\bullet,\bullet}) \right) \right. \nonumber \\[5pt]
 && \left. + \sum_{p+q=k+1} \left( \dim_\mathbb{K} H^{p,q}_{\partial} (B^{\bullet,\bullet}) + \dim_\mathbb{K} H^{p,q}_{\overline\partial}(B^{\bullet,\bullet}) \right)\right) \;, \nonumber 
 \end{eqnarray}
 and
 \begin{eqnarray}\label{eq:upper-bound-bottchern-algebraic}
  \lefteqn{ \sum_{p+q=k} \dim_\mathbb{K} H^{p,q}_{BC}(B^{\bullet,\bullet}) } \\[5pt]
 &\leq& (N+1) \cdot \left( \sum_{p+q=k} \left( \dim_\mathbb{K} H^{p,q}_{\partial}(B^{\bullet,\bullet}) + \dim_\mathbb{K} H^{p,q}_{\overline\partial}(B^{\bullet,\bullet}) \right) \right. \nonumber \\[5pt]
 && \left. + \sum_{p+q=k-1} \left( \dim_\mathbb{K} H^{p,q}_{\partial} (B^{\bullet,\bullet}) + \dim_\mathbb{K} H^{p,q}_{\overline\partial}(B^{\bullet,\bullet}) \right)\right) \;. \nonumber 
 \end{eqnarray}
\end{thm}

The proof goes exactly as in Theorem \ref{thm:upper-bound}. Note that we have to assume that there is a vertical strip of finite uniform width and shifted in the horizontal non-negative direction such that the double complex has support in this strip. In particular, this implies that the Fr\"olicher spectral sequences associated to the natural filtrations of the double complex degenerate in a finite number of steps, depending on $N$. Note also that the conjugate-Dolbeault cohomology $H^{\bullet,\bullet}_{\partial}(B^{\bullet,\bullet})$ and the Dolbeault cohomology $H^{\bullet,\bullet}_{\overline\partial}(B^{\bullet,\bullet})$ are no longer isomorphic, whence we need to consider both of them in the right-hand side of the inequalities \eqref{eq:upper-bound-aeppli-algebraic} and \eqref{eq:upper-bound-bottchern-algebraic}.

\subsection{Upper bound for symplectic Bott-Chern cohomology}
As a consequence, we can apply the above result to symplectic geometry. Let $X$ be a compact manifold of dimension $2n$ endowed with a symplectic structure $\omega$. Consider the exterior differential $d$ and the {\em symplectic co-differential} $d^\Lambda := [ d, -\iota_{\omega^{-1}} ]$. They satisfy $[d,d^\Lambda]=0$. Then, as in \cite{brylinski, cavalcanti-phd}, define the double complex
$$ \left( B^{\bullet_1,\bullet_2} \;:=\; \wedge^{\bullet_1-\bullet_2}X \otimes \beta^{\bullet_2} ,\; d\otimes\mathrm{id},\; d^\Lambda\otimes\beta \right) \;, $$
where $\beta$ is a generator of the infinite cyclic commutative group $\beta^\Z$. Note that, for any $q\in\Z$, we have
$$ B^{p,q} \;=\; \{0\} \qquad \text{ for } p\not\in\{q, \ldots, q+2n \} \;, $$
hence there exists a vertical strip as in the statement of Theorem \ref{eq:upper-bound-bottchern-algebraic} of width $N+1$ with $N:=2n$ such that the double complex $B^{\bullet,\bullet}$ has support in this strip.
The Bott-Chern and Aeppli cohomologies of $B^{\bullet,\bullet}$ are related to the symplectic cohomologies
$$ H^{\bullet}_{d+d^\Lambda}(X) \;:=\; \frac{\ker d \cap \ker d^{\Lambda}}{\imm d d^\Lambda} \qquad \text{ and } \qquad H^{\bullet}_{dd^\Lambda}(X) \;:=\; \frac{\ker dd^\Lambda}{\imm d+\imm d^\Lambda} $$
introduced and studied by L.-S. Tseng and S.-T. Yau, \cite{tseng-yau-1, tseng-yau-2}: more precisely,
$$
H^{\bullet_1,\bullet_2}_{BC}(B^{\bullet,\bullet}) \;=\; H^{\bullet_1-\bullet_2}_{d+d^\Lambda}(X) \otimes \beta^{\bullet_2}
\qquad \text{ and } \qquad
H^{\bullet_1,\bullet_2}_{A}(B^{\bullet,\bullet}) \;=\; H^{\bullet_1-\bullet_2}_{dd^\Lambda}(X) \otimes \beta^{\bullet_2} \;.
$$
The conjugate-Dolbeault and Dolbeault cohomologies of $B^{\bullet,\bullet}$ are both related to the de Rham cohomology of $X$.

In \cite[Theorem 4.4]{angella-tomassini-5}, it is proven that, for any $k\in\Z$, the inequality
$$ \dim_\R H^k_{d+d^\Lambda}(X) + \dim_\R H^k_{dd^\Lambda}(X) \;\geq\; 2\, \dim_\R H^k_{dR}(X;\R) \;, $$
and that the equality holds for any $k\in\Z$ if and only if $X$ satisfies the Hard Lefschetz Condition, equivalently, the $dd^\Lambda$-Lemma. 

By applying Theorem \ref{thm:upper-bound-algebraic}, we get the following upper bound.
(Note also that, with the same trick, analogous results can be obtained for generalized-complex structures in the sense of N. Hitchin \cite{hitchin}.)

\begin{thm}\label{thm:upper-bound-symplectic}
 Let $X$ be a compact differentiable manifold of dimension $2n$ endowed with a symplectic structure $\omega$. Then, for any $k\in\Z/2\Z$,
 $$
 \sum_{h = k \,\mathrm{mod}\, 2} \dim_\R H^{h}_{d+d^\Lambda}(X)
 \;\leq\;
 2(2n+1) \cdot \sum_{h\in\Z} \dim_\R H^{h}_{dR}(X;\R) \;,
 $$
 and
 $$
 \sum_{h = k \,\mathrm{mod}\, 2} \dim_\R H^{h}_{dd^\Lambda}(X)
 \;\leq\;
 2(2n+1) \cdot \sum_{h\in\Z} \dim_\R H^{h}_{dR}(X;\R) \;.
 $$
\end{thm}

\begin{rmk}
 Note that Theorem \ref{thm:char-deldelbar-minus} does not admit a generalization to symplectic manifolds. In fact, for any compact symplectic manifold $X$, for any $k\in\Z$, it holds $\dim_\R H^{k}_{d+d^\Lambda}(X)=\dim_\R H^{k}_{dd^\Lambda}(X)$ by \cite[Proposition 3.24]{tseng-yau-1}, as a consequence of the validity of the Hard Lefschetz Condition on symplectic Bott-Chern and Aeppli cohomologies.
\end{rmk}

\end{document}